\documentclass{article}
\usepackage[a4paper]{geometry}
\usepackage{amsmath}
\usepackage{amssymb}
\usepackage{amsthm}
\usepackage{mathrsfs}
\usepackage{mathtools}

\usepackage{tikz}
\usepackage{cleveref}
\usepackage{tikz-cd}
\usepackage{mathtools}

\usepackage{placeins}

\usepackage{float}

\usepackage{comment}

\usepackage{enumerate}

\theoremstyle{definition}
\newtheorem{thm}{Theorem}[section]
\crefname{thm}{Theorem}{Theorems}
\newtheorem{cor}[thm]{Corollary}

\crefname{prop}{Proposition}{Propositions}
\newtheorem{lem}[thm]{Lemma}
\crefname{lem}{Lemma}{Lemmas}
\newtheorem{clm}[thm]{Claim}

\crefname{defn}{Definition}{Definitions}

\newtheorem{rmk}[thm]{Remark}

\newtheorem*{ack*}{Acknowledgements}

\newcommand{\co}{\operatorname{co}}

\usepackage{titling}

\title{The sharp doubling threshold for approximate convexity.}
\author{Peter van Hintum\thanks{New College, University of Oxford, UK. email: peter.vanhintum@new.ox.ac.uk}, Peter Keevash\thanks{Mathematical Institute, University of Oxford, UK. Supported by ERC Advanced Grant 883810.}}
\allowdisplaybreaks

\begin{document}
\maketitle
\begin{abstract}
We show for $A,B\subset\mathbb{R}^d$ of equal volume and $t\in (0,1/2]$ that if $|tA+(1-t)B|< (1+t^d)|A|$, then (up to translation) $|\co(A\cup B)|/|A|$ is bounded. This establishes the sharp threshold for Figalli and Jerison's quantative stability of the Brunn-Minkowski inequality. We additionally establish a similar sharp threshold for iterated sumsets.
\end{abstract}

\section{Introduction}
The Brunn-Minkowski inequality asserts that for sets $A,B\subset\mathbb{R}^d$ with equal volume and $t\in(0,1/2]$, we have
$$|tA+(1-t)B|\geq |A|,$$
with equality exactly if $A=B$ is a convex set. The stability of this inequality has sparked a rich body of research (e.g. \cite{figalli2009refined,figalli2010mass,christ2012near,eldan2014dimensionality,figalli2015quantitative,figalli2017quantitative,Euclidean,figalli2021quantitative,van2021sharp,van2020sharp,planarBM}). These results variously control (up to translation) $|A\triangle B|$, $|\co(A)\setminus A|$, and $|\co(A\cup B)\setminus A|$ in terms of the parameter:
$$\delta_t(A,B):=\frac{|tA+(1-t)B|}{|A|}-1\geq 0.$$

In \cite{figalli2017quantitative}, Figalli and Jerison showed that there exist $a_{d,t},c_{d,t},\Delta_{d,t}>0$, so that if $\delta=\delta_t(A,B)\leq \Delta_{d,t}$, then (up to translation)
$$|\co(A\cup B)\setminus A|\leq c_{d,t} \delta^{a_{d,t}}|A|.$$
For various particular classes of sets $A,B$ the optimal values of $a_{d,t}$ and $c_{d,t}$ have been determined. The optimal values for general $A,B\subset\mathbb{R}^d$ are expected to be $a_{d,t}=1/2$ and $c_{d,t}=O(t^{-1/2})$, though for $A=B$ the stronger result with $a_{d,t}=1$ and $c_{d,t}=t^{-1} \exp(O(d\log(d)))$ has been established \cite{van2021sharp}.

In this paper we determine the optimal value of $\Delta_{d,t}$ for this result. We establish this bound both for general $A,B\subset\mathbb{R}^d$ and for iterated sumsets. Both can be extended to quantative stability results for all doublings below this threshold.
\begin{thm}\label{distinctsets}
For all $d\in\mathbb{N},t\in(0,1/2]$, there are $C_{d,t}>0$ so that if $A,B\subset \mathbb{R}^d$ of the same volume have $|tA+(1-t)B|< (1+t^d)|A|$, then (after possibly translating) $|\co(A\cup B)|\leq C_{d,t}|A|$.
\end{thm}
In fact, we can choose $C_{d,t}=t^{-O(d^2)}$.
The second theorem determines this threshold for iterated sumsets. For $X\subset\mathbb{R}^d$ and $k\in\mathbb{N}$, we write $k\cdot X:=\underbrace{X+\cdots+X}_{k \text{ terms}}$.

\begin{thm}\label{iteratedsumsets}
For all $d,k\in\mathbb{N}$, there are $C_{d,k}>0$ so that if $A\subset \mathbb{R}^d$ satisfies $|k\cdot A|< (1^d+\dots+k^d)|A|$, then $|\co(A)|\leq C_{d,k}|A|$.
\end{thm}

\begin{rmk}
Cole Hugelmeyer, Hunter Spink, and Jonathan Tidor established \Cref{iteratedsumsets} independently. \Cref{distinctsets} for $t=1/2$ coincides with \Cref{iteratedsumsets} for $k=2$. We establish this result as well as \Cref{linearstabilitycor} for $t=1/2$ through independent methods in \cite[Corollary 1.9]{LocalityinSumsets}. 
\end{rmk}

Combining \Cref{distinctsets} with the main result from \cite{figalli2017quantitative} (included here as \Cref{figjer}), we find the following quantitative stability of the Brunn-Minkowski inequality.

\begin{cor}\label{figjercor}
For all $d\in\mathbb{N},t\in(0,1)$, there exist $a_{d,t},C_{d,t}>0$ so that if $A,B\subset \mathbb{R}^d$ of the same volume satisfy $\delta:=\delta_t(A,B)< t^d$, then (up to translation) we have 
$$|\co(A\cup B)\setminus A|\leq C_{d,t}\delta^{a_{d,t}}|A|.$$
\end{cor}

Combining \Cref{distinctsets} with the main result from \cite{van2021sharp} (included as \Cref{homobm}) proving a conjecture from \cite{figalli2021quantitative}, we find the following stability of the Brunn-Minkowski inequality for homothetic sets.

\begin{cor}\label{linearstabilitycor}
For all $d\in\mathbb{N},t\in(0,1/2]$, there exist $C_{d,t}$ so that if $A\subset \mathbb{R}^d$ satisfies $\delta:=\delta_t(A,A)< t^d$, then
$$|\co(A)\setminus A|\leq C_{d,t}\delta|A|.$$
\end{cor}

All these results are sharp as shown by $A=[0,1]^d$ and $B=A\cup \{v\}$, where $v\in\mathbb{R}^d$ is some arbitrarily large vector. For these $A,B$, we have 
$$tA+(1-t)B=A \cup [0,t]^d+(1-t)v,$$
so that $\delta_t(A,B)=t^d$, while $\frac{|\co(A)\setminus A|}{|A|}\to \infty$ as $||v||_2\to \infty$.

In \Cref{1dsec}, we establish strong versions of the results for $d=1$ which are instrumental in the proof of the general results. In \Cref{distinctsec}, we prove \Cref{distinctsets}. In \Cref{iteratedsec}, we prove \Cref{iteratedsumsets}. Finally, in \Cref{corsec} we include proofs of the corollaries.

The idea in both proofs is to find two points in $A$ (or $B$) for each coordinate direction that are very far apart, which can be done by increasing $|\co(A)|$ in combination with \Cref{positioning}. We then distinguish two cases; either $A$ contains long fibres in all coordinate directions or not. In the former case, we find a lower bound on the doubling using Pl\"unnecke's inequality\footnote{Pl\"unnecke's inequality \cite{Plunnecke} states that $|X+Y|\leq \lambda |Y|$ implies $|d\cdot X|\leq \lambda^d|Y|$.} as the sum of those long fibres is large (see \Cref{longfibresinalldirections}). In the latter case we fix a direction in which the fibres of $A$ are short and show that using an optimal transport map, we can pair up the fibres from $A$ and $B$ whose (weighted) sum form a reference set of size $|A|$ (see \Cref{transportlem}). Finally we show that summing fibres of $B$ with the two far removed points from $A$ gives a set disjoint from the reference set of the required size (see \Cref{distinctsetlem}).

\section{Strong versions of Freiman's $3k-4$ theorem}\label{1dsec}
We use two versions of the following lemma. This lemma implies continuous versions of Freiman's $3k-4$ theorem for distinct sets.

\begin{lem}\label{distinctsetlem}
Given subsets $X,Y,Z\subset [0,1]$, we have 
$$|(X+Y) \cup (\{0,1\}+Z)|\geq \min\{1,|X|+|Y|\}+|Z|.$$
\end{lem}

\begin{proof}
Let $S:=X+Y \cup \{0,1\}+Z$. Let $f:\mathbb{R}\to\mathbb{T}; x\mapsto x-\lfloor x\rfloor$ be the canonical quotient map from the line to the torus. Note that for $z\in Z$, we have $|f^{-1}(z)\cap S|\geq 2$. Moreover, note that $|f(X+Y)|\geq \min\{1,|f(X)|+|f(Y)|\}$ by Cauchy-Davenport, so we find:
\begin{align*}|S|&\geq |S\cap f^{-1}(Z)|+|S\setminus f^{-1}(Z)|\\
&\geq|(\{0,1\}+Z)\cap f^{-1}(Z)|+|(X+Y)\setminus f^{-1}(Z)|\\
&\geq 2|Z| + |f(X+
Y)\setminus Z|\\
&\geq 2|Z|+\big(\min\{1,|X|+|Y|\}- |Z|\big)\\
&=\min\{1,|X|+|Y|\}+|Z|.
\end{align*}
The lemma follows.
\end{proof}
We will only apply \Cref{distinctsetlem} with sets so that $|X|,|Y|,|Z|\leq \frac12$, so that the bound gives $|X|+|Y|+|Z|$.

Note that for sets $A,B\subset\mathbb{R}$ with $|\co(A)|\geq |\co(B)|$, we can set $X=A$ and $Y=Z=B$ so that
$$|A+B|\geq |B|+\min\{|\co(A)|,|A|+|B|\}=|A|+|B|+\min\{|\co(A)\setminus A|, |B|\},$$
which can be seen as a stronger version of the one dimensional instance of \Cref{distinctsets}.

For iterated sumsets we have the following version of this lemma. 

\begin{lem}\label{strengthenedlem}
Let $Y_i\subset [0,1]$ (for $1\leq i\leq k$) so that $|Y_i|\leq 1/k$, and let $S:=\bigcup_{i=1}^k \{0,1, \dots, (k-i)\}+i\cdot Y_i$, then we have $|S|\geq  \sum_i i|Y_i|$
\end{lem}

\begin{proof}
Let $f:\mathbb{R}\to\mathbb{T}; x\mapsto x-\lfloor x\rfloor$ be the canonical quotient map from the line to the torus. Note that for $y\in f(i\cdot Y_i)$, we have $|f^{-1}(y)\cap S|\geq k-i+1$. Moreover, note that $|f(i\cdot Y_i)|\geq i|Y_i|$ by Cauchy-Davenport.
Let $Z_i:=f(i\cdot Y_i)\setminus \bigcup_{j<i}f\left(j\cdot Y_j\right)$. With a little thought (e.g. by induction on $k$), we find  $\sum_i (k-i+1) |Z_i|\geq\sum_i |f(i\cdot Y_i)|\geq \sum_i i|Y_i|$.
Combining these we find
\begin{align*}|S|&\geq \sum_i|S\cap f^{-1}(Z_i)|\\
&\geq \sum_i \left|(\{0,1,\dots, (k-i)\}+i\cdot Y_i)\cap f^{-1}(Z_i)\right|\\
&\geq \sum_i (k-i+1)|Z_i|\\
&\geq \sum_i i|Y_i|.
\end{align*}
The lemma follows.
\end{proof}

This proof gives the following, which can be seen as the continuous version of Corollary 1 from \cite{lev1996structure}. For $A\subset \mathbb{R}$ if we let $\ell:=\min\left\{\left\lfloor\frac{|\co(A)|}{|A|}\right\rfloor, k\right\}$, then 
$$|k\cdot A|\geq \binom{\ell+1}{2}|A|+\left(k-\ell\right)|\co(A)|.$$
In the most dense situation (i.e. $\ell=1$), this gives that if $|k\cdot A|\leq (k+1)|A|$, then $|\co(A)\setminus A|\leq \frac{1}{k-1}\left(|k\cdot A|-k|A|\right)$. In the least dense situation (i.e. $\ell=k$), this is a sharp version of \Cref{iteratedsumsets} in one dimension. For $k=2$ this reduces to the continuous version of Freiman's $3k-4$ theorem. These results are sharp as shown by a union of an interval with a point.

\section{Distinct sets: \Cref{distinctsets}}\label{distinctsec}

We use the following standard lemma that establishes the existence of a large subset of $tA+(1-t)B$. For an exposition see e.g. Section 3, Step 1 in \cite{figalli2015stability}, for a proof of this specific lemma see Appendix D in \cite{mccann1994convexity}.

\begin{lem}[\cite{mccann1994convexity}]\label{transportlem}
Let $\mu_A,\mu_B\colon\mathbb{R}^{d-1}\to \mathbb{R}$ be two probability measures and $T:\mathbb{R}^{d-1}\to\mathbb{R}^{d-1}$ the optimal transport map so that for all measurable $X\subset \mathbb{R}^{d-1}$, we have $\mu_A(X)=\mu_B(T(X))$. For $t\in(0,1)$, let $\rho_t:\mathbb{R}^{d-1}\to \mathbb{R}$ be defined by
$\rho_t(x):= t\mu_A(y)+(1-t)\mu_B(T(y))$ where $y\in\mathbb{R}^{d-1}$ is the unique element so that $x=ty+(1-t)T(y)$. Then $\int \rho_t\geq 1$.
\end{lem}

\begin{lem}\label{positioning}
    Given convex sets $X,Y\subset\mathbb{R}^d$, there exists a translation $v\in\mathbb{R}^d$ and an affine transformation $T:\mathbb{R}^d\to\mathbb{R}^d$, so that if we let $U:=T(X)$ and $V:=v+T(Y)$, then there are points $p^i\in U\cup V$, $\lambda_i\in\mathbb{R}^d$ and hyperplanes $H_i\subset\mathbb{R}^d$ (for $1\leq i\leq d$) so that:
    \begin{enumerate}
        \item if $p^i\in U$, then $p^i+\lambda_ie_i\in U$ and if  $p^i\in V$, then $p^i+\lambda_ie_i\in V$,
        \item $U\cup V\subset H_i+[0,\lambda_i] e_i$, and
        \item $\big|\bigcap_i H_i+[0,\lambda_i] e_i\big|= \prod_i \lambda_i$
    \end{enumerate}
\end{lem}

\begin{proof}
We proceed by induction on $k$; assume after an affine transformation and translation, there are (for $1\leq i\leq k$) $p^i\in X\cup Y$, $\lambda_i\in\mathbb{R}^d$, and hyperplanes $H_i\subset\mathbb{R}^d$ so that:
\begin{enumerate}
        \item if $p^i\in X$, then $p^i+\lambda_ie_i\in X$ and if  $p^i\in Y$, then $p^i+\lambda_ie_i\in Y$,
        \item $X\cup Y\subset H_i+[0,\lambda_i] e_i$,
        \item $\left|\mathbb{R}^{k}\times \{0\}^{d-k}\cap \bigcap_{1\leq i\leq k} H_i+[0,\lambda_i] e_i\right|= \prod_{1\leq i\leq k} \lambda_i$
    \end{enumerate}
Translating one of the sets along a multiple of $e_{k+1}$ we may assume that the points $q,r\in X\cup Y$ minimizing $\langle e_{k+1},q\rangle$ and maximizing $\langle e_{k+1},r\rangle$ belong to the same set $X$ or $Y$. Note that translating along $e_{k+1}$ does not affect any of the properties of $p^i,\lambda_i$, and $H_i$ with $i\leq k$ (up to the appropriate translations). For notational convenience translate both sets by $-q$ (i.e. assume $q=0$). Let $H_{k+1}:=\mathbb{R}^{k}\times \{0\}\times \mathbb{R}^{d-k-1}$ and $\lambda_{k+1}:=\langle e_{k+1},r\rangle$. Clearly, we have $X\cup Y\subset H_{k+1}+[0,\lambda_{k+1}] e_{k+1}$ and 
$$\Bigg|\mathbb{R}^{k+1}\times \{0\}^{d-k-1}\cap \bigcap_{1\leq i\leq k+1} H_i+[0,\lambda_i] e_i\Bigg|=\Bigg|\mathbb{R}^{k}\times \{0\}^{d-k}\cap \bigcap_{1\leq i\leq k} H_i+[0,\lambda_i] e_i\Bigg|\cdot \left|\lambda_{k+1}e_{k+1}\right|= \prod_{1\leq i\leq k+1} \lambda_i.$$
The only issue is that though $r\in H_{k+1}+\lambda_{k+1}e_{k+1}$, it might not coincide with $\lambda_{k+1}e_{k+1}$. Hence, we apply the affine transformation $T\colon \mathbb{R}^d\to\mathbb{R}^d, x\mapsto x-\langle x,e_{k+1}\rangle (\lambda_{k+1}^{-1}r-e_{k+1})$ to $X,Y,p^i$, and $H_i$ (for $i\leq k$). $T$ preserves all planes parallel to $\mathbb{R}^{k}\times \{0\}\times \mathbb{R}^{d-k-1}=H_{k+1}$, and takes $r$ to $\lambda_{k+1}e_{k+1}$. As the basis vectors $e_i$ ($i\neq k+1$) are preserved by $T$, the inductive hypotheses are not affected. Choosing $p^{k+1}=(0,\dots,0)$ concludes the induction.
\end{proof}

\begin{proof}[Proof of \Cref{distinctsets}]
Let $C_{d,t}:=L^d$ where $L=L_{d,t}:=\left(\frac{2}{(1-t)t}\right)^{4d}$. We'll prove the contrapositive, so let $|\co(A\cup B)|\geq C_{d,t}|A|$. Normalise so that $|A|=|B|=1$.

For $i=1,\dots d$, let $\pi^i:\mathbb{R}^d\to \mathbb{R}$ be the coordinate projections and $\pi_i\colon\mathbb{R}^d\to\mathbb{R}^{d-1}$ the complementary projections.

First, apply \Cref{positioning} to $\co(tA)$ and $\co((1-t)B)$ and apply an affine transformation so that all $\lambda_i$ are equal and at least $L$. For $i=1,\dots, d$, let $x^i\in\mathbb{R}^{d-1}$ so that 
\begin{align*}
\max\left\{|\co(\pi_i^{-1}(x^i)\cap tA)|,|\co(\pi_i^{-1}(x^i)\cap (1-t)B)|\right\}&=\max\left\{|\pi_i^{-1}(x^i)\cap \co(tA)|,|\pi_i^{-1}(x^i)\cap \co((1-t)B)|\right\}\\
&=\max_{x\in\mathbb{R}^{d-1}}\left\{|\pi_i^{-1}(x)\cap \co(tA)|,|\pi_i^{-1}(x)\cap \co((1-t)B)|\right\}\\
&=L'>L.
\end{align*}
Let $a^i,b^i\in \mathbb{R}^d$ be such that 
$$|\pi_i^{-1}(a^i)\cap tA|\geq \max_{a\in\mathbb{R}^d}|\pi_i^{-1}(a)\cap tA|\text{ and }|\pi_i^{-1}(b^i)\cap (1-t)B|\geq \max_{b\in\mathbb{R}^d}|\pi_i^{-1}(b)\cap (1-t)B|.$$

\begin{clm}\label{longfibresinalldirections}
If for all $i=1,\dots,d$, we have $\max \left\{|\pi_i^{-1}(a^i)\cap tA|,|\pi_i^{-1}(b^i)\cap (1-t)B|\right\}\geq \sqrt{L}$, then $|tA+(1-t)B|\geq 2|A|$.
\end{clm}
\begin{proof}
For a contradiction assume $|tA+(1-t)B|< 2|A|$. Distinguish two cases, either for all $i$, $$\min \left\{|\pi_i^{-1}(a^i)\cap tA|,|\pi_i^{-1}(b^i)\cap (1-t)B|\right\}\geq 1$$ or not.

In the latter case, consider the $i_0$ so that $\min \left\{|\pi_{i_0}^{-1}(a^{i_0})\cap tA|,|\pi_{i_0}^{-1}(b^{i_0})\cap (1-t)B|\right\}\leq 1$. Assume $|\pi_{i_0}^{-1}(a^{i_0})\cap tA|\leq1$ (the other case follows analogously). As $|tA|=t^d$, this implies $|\pi_{i_0}(tA)|\geq t^{d}$, so that
$$|tA+(1-t)B|\geq |\pi_{i_0}(tA)|\cdot \left|\left(\pi^{-1}_{i_0}(b^{i_0})\cap (1-t)B\right)\right|\geq t^{d}\sqrt{L}>2,$$
a contradiction.

Hence, we may assume $\min \left\{|\pi_i^{-1}(a^i)\cap tA|,|\pi_i^{-1}(b^i)\cap (1-t)B|\right\}\geq 1$ for all $i$. Note that by Pl\"unnecke's inequality $|tA+(1-t)B|< (1+(1-t)^d)|A|<2t^{-d}|tA|$ implies that 
$$\big|d\cdot (1-t)B\big|\leq (2t^{-d})^d|tA|<\left(\frac{2}{(1-t)t}\right)^{d^2}.$$
Analogously we find the same bound on $\big|d\cdot A\big|$. On the other hand, we find
$$\big|d\cdot (1-t)B\big|\geq \left|\sum_{i=1}^d \left(\pi_i^{-1}(b^i)\cap (1-t)B\right)\right|=\prod_{i=1}^d\left|\pi_i^{-1}(b^i)\cap (1-t)B\right|,$$
and the analogous result for $tA$. Combining these bounds on the iterated sumsets of $tA$ and $(1-t)B$, we find
$$\left(1 \cdot \sqrt{L}\right)^d\leq \prod_{i=1}^d\left|\pi_i^{-1}(a^i)\cap tA\right|\cdot \left|\pi_i^{-1}(b^i)\cap (1-t)B\right|\leq  \big|d\cdot tA\big|\cdot \big|d\cdot (1-t)B\big|< \left(\frac{2}{(1-t)t}\right)^{2d^2},$$
a contradiction.
\end{proof}

Hence, we may assume there is a coordinate direction $i$ with $|\pi_i^{-1}(y)\cap tA|,|\pi_i^{-1}(y)\cap (1-t)B|\leq \sqrt{L}$ for all $y\in\mathbb{R}^{d-1}$. Rotating if necessary, we may assume $i=1$. For notational convenience, write $A_x:=\pi_1^{-1}(x)\cap A$, $B_x:=\pi_1^{-1}(x)\cap B$, and $S_x:=\pi_1^{-1}(x)\cap (tA+(1-t)B)$.

By the application of \Cref{positioning}, we have that at least one of $|\co(tA_{x^1})|$ and $|\co((1-t)B_{x^1})|$ is larger than $L>2 \max_{x\in \mathbb{R}^{d-1}}\{t|A_x|,(1-t)|B_x|\}$. Henceforth assume the latter\footnote{Though the other case follows analogously, there is an asymmetry between $t$ and $1-t$ that gives a stronger result in the other case.}.

Translate $B$ so that $x^1=0$ and $(1-t)\co(B_0)=[0,L']\times (0,\dots,0)$.

Write $\mu_A\colon\mathbb{R}^{d-1}\to\mathbb{R}, x\mapsto |A_x|$ and $\mu_B$ analogously.
Let $T\colon\mathbb{R}^{d-1}\to\mathbb{R}^{d-1}$ be the optimal transport map that takes $\mu_A$ to $\mu_B$. Define $$S^1:=\bigcup_{x\in\pi(A)}tA_x(+)(1-t)B_{T(x)},$$ where for two subsets of $A_x,B_y\subset\mathbb{R}^{d-1}$, 
$$tA_x(+)(1-t)B_y:=(t\min(I)+(1-t)J)\cup (tI+(1-t)\max(J))\times \{tx+(1-t)y\},$$
where $I,J\subset\mathbb{R}$ are so that $A_x=I\times \{x\}$ and $B_y=J\times \{y\}$. Note that $|S^1_x|=t\mu_A(y)+(1-t)\mu_B(y)$ where $y\in\mathbb{R}^{d-1}$ is the unique element with $x=ty+(1-t)T(y)$. Hence, by \Cref{transportlem}, we find $|S^1|=\int |S^1_x|dx\geq |A|$.

Define $S^2:=tA+\left\{(0,\dots,0),(L',0,\dots,0)\right\}$ and note that $S^1,S^2\subset tA+(1-t)B$.

For any $x\in\mathbb{R}^{d-1}$, we find that
$$S_{tx}\supset S^1_{tx}\cup S^2_{tx}= tA_{y}+(1-t)B_{T(y)}\cup  tA_x+\left\{(0,\dots,0),(L',0,\dots,0)\right\}$$ 
where $y$ is such that $ty+(1-t)T(y)=tx$ (if such a $y$ exists)
which by \Cref{distinctsetlem} implies that
$$|S_{tx}\setminus S^1|\geq |tA_{x}|.$$
Integrating over all $x$ we find
\begin{align*}
|tA+(1-t)B\setminus S^1|&=\int_{x\in\mathbb{R}^{d-1}} |S_x\setminus S^1|\\
&\geq t^{d-1} \int_{x\in\mathbb{R}^{d-1}} |S_{tx}\setminus S^1|\\
&\geq t^{d-1} \int_{x\in\mathbb{R}^{d-1}} |tA_x|\\
&= t^{d} \int_{x\in\mathbb{R}^{d-1}} |A_x|\\
&= t^{d}|A|.
\end{align*}
This concludes the proof of the theorem.
\end{proof}

\section{Iterated sumsets: \Cref{iteratedsumsets}}\label{iteratedsec}

\begin{proof}[Proof of \Cref{iteratedsumsets}]
Let $L=L_{d,k}$ be sufficiently large in terms of $k$ and $d$.

First, apply \Cref{positioning} to $X=Y=\co(A)$ and for $i=1,\dots, d$, let $x^i\in\mathbb{R}^{d-1}$ so that $|\pi_i^{-1}(x^i)\cap \co(A)|=\max_{x\in\mathbb{R}^{d-1}}|\pi_i^{-1}(x)\cap \co(A)|=L'>L$.
\begin{clm}
If for all $i$, there exists a $y_i\in\mathbb{R}^{d-1}$ with $|\pi_i^{-1}(y_i)\cap A|\geq L/k$, then $|k\cdot A|\geq (1^d+2^d+\dots+k^d)|A|$
\end{clm}
\begin{proof}
For a contradiction assume  $\big|k\cdot A\big|< (1^d+2^d+\dots+k^d) |A|$. Then by Pl\"unnecke's inequality we have $\big|d\cdot A\big|< (1^d+2^d+\dots+k^d)^d |A|$. However, $$\big|d\cdot A\big|\geq \left|\sum_{i=1}^d \left(\pi_i^{-1}(y_i)\cap A\right)\right|\geq (L/k)^d>(1^d+2^d+\dots+k^d)^d.$$
This contradiction proves the claim.
\end{proof}
Hence, we may assume there is a coordinate direction $i$ with $|\pi_i^{-1}(y_i)\cap A|\leq L/k$ for all $y_i\in\mathbb{R}^{d-1}$. Rotating if necessary, we may assume $i=1$. For notational convenience, write $A_x:=\pi_1^{-1}(x)\cap A$ and $S_x:=\pi_1^{-1}(x)\cap (A+A)$.

Translate $A$ so that $x^1=0$, and $A_0\supset\{0,L'\}\times (0,\dots,0)$. Now we find that 
$$S_{x}\supset \bigcup_{i=1}^k i\cdot A_{x/i}+(k-i)\cdot A_0$$ 
which by \Cref{strengthenedlem} implies that 
$$|S_{x}|\geq \sum_{i=1}^k i |A_{x/i}|.$$
We conclude:
\begin{align*}
|A+A|&=\int_{x\in\mathbb{R}^{d-1}} |S_x|dx\\
&\geq  \int_{x\in\mathbb{R}^{d-1}} \sum_{i=1}^k i |A_{x/i}|dx\\
&=\sum_{i=1}^k  i^d\int_{x\in\mathbb{R}^{d-1}} |A_{x}|dx\\
&= (1^d+\dots+k^d)|A|.
\end{align*}
This concludes the proof of the theorem.
\end{proof}

\section{Proofs of the corollaries}\label{corsec}
First recall the main theorems from \cite{figalli2017quantitative} and \cite{van2021sharp}.
\begin{thm}[\cite{figalli2017quantitative}]\label{figjer}
For all $d\in\mathbb{N},t\in(0,1)$, there exist $a_{d,t},C_{d,t},\Delta_{d,t}>0$ so that if $A,B\subset \mathbb{R}^d$ of the same volume satisfy $\delta:=\delta_t(A,B)< \Delta_{d,t}$, then (up to translation) we have 
$$|\co(A\cup B)\setminus A|\leq C_{d,t}\delta^{a_{d,t}}|A|.$$
\end{thm}
 \begin{thm}[\cite{van2021sharp}]\label{homobm}
For all $d\in\mathbb{N},t\in(0,1/2]$, there exist $C_{d}, \Delta_{d,t}>0$ so that if $A\subset \mathbb{R}^d$ satisfies $\delta:=\delta_t(A,A)< \Delta_{d,t}$, then
$$|\co(A)\setminus A|\leq C_{d}t^{-1}\delta|A|.$$
\end{thm}

The corollaries follow quickly.

\begin{proof}[Proof of \Cref{figjercor}]
Let $C_{d,t}:=\max\{C^{\ref{distinctsets}}_{d,t}(\Delta^{\ref{figjer}}_{d,t})^{-a_{d,t}}, C^{\ref{figjer}}_{d,t}\}$ and $a_{d,t}:=a_{d,t}^{\ref{figjer}}$, where $a_{d,t}^{\ref{figjer}},C^{\ref{figjer}}_{d,t},$ and $\Delta^{\ref{figjer}}_{d,t}$ are the constants from \Cref{figjer} and $C^{\ref{distinctsets}}_{d,t}$ is the constant from \Cref{distinctsets}.

Distinguish two cases; either $\delta<\Delta^{\ref{figjer}}_{d,t}$ or $\Delta^{\ref{figjer}}_{d,t}\leq \delta <t^d$. In the former case \Cref{figjer} gives (after translation)
$$|\co(A\cup B)\setminus A|\leq C^{\ref{figjer}}_{d,t}\delta^{a_{d,t}^{\ref{figjer}}}|A|\leq C_{d,t}\delta^{a_{d,t}}|A|.$$
In the latter case, we find by \Cref{distinctsets} that
$$|\co(A\cup B)\setminus A|\leq |\co(A\cup B)|\leq C^{\ref{distinctsets}}_{d,t}|A|\leq C_{d,t}(\Delta^{\ref{figjer}}_{d,t})^{a_{d,t}}|A|\leq C_{d,t}\delta^{a_{d,t}}|A|.$$
Combining the cases gives the corollary.
\end{proof}

The proof of \Cref{linearstabilitycor} follows similarly.

\section*{Acknowledgements}
The first author would like to thank Anne de Roton and Pablo Candela for focusing his attention on this problem.

\bibliographystyle{alpha}
\bibliography{references}

\end{document}